\documentclass[11pt]{article}
	\usepackage{amsfonts,amssymb,amsthm,amsmath,cite,bbm,textcomp}

	\newcommand{\R}{\mathbb R}  
	\newcommand{\E}{\mathbb E} 
	
	\renewcommand{\S}{\mathbb S}
	\renewcommand{\H}{\mathbb H} 
	\newcommand{\B}{\mathbb B}
	
	\newcommand{\cK}{{\mathcal K}} 
	\newcommand{\cP}{{\mathcal P}}

	\newcommand{\sn}{{\scriptscriptstyle \S^n}}
	\newcommand{\hn}{{\scriptscriptstyle \H^n}}
	\newcommand{\ps}{{\scriptscriptstyle +}}
	\newcommand{\mss}{{\scriptscriptstyle -}}
	\newcommand{\pms}{{\scriptscriptstyle \pm}}
	\newcommand{\ms}{{\scriptscriptstyle -}}
	\newcommand{\rn}{{\R^{n}}}
	
	\newcommand{\Int}{\operatorname{int}}
	\newcommand{\dist}{\operatorname{dist}}
	\newcommand{\bd}{\partial}
	\newcommand{\ldel}{\operatorname{ldel}}
	\newcommand{\ldiv}{\operatorname{ldiv}}
	\newcommand{\vol}{\operatorname{vol}}
	\newcommand{\Vol}{\operatorname{V}_n}

	\newcommand{\Hs}{H_{n-1}^{\scriptscriptstyle\S^{\scriptscriptstyle n}}}
	\newcommand{\Hh}{H_{n-1}^{\scriptscriptstyle\H^{\scriptscriptstyle n}}}
	\newcommand{\eqnref}[1]{(\ref{#1})} 
	
	\newcommand{\gln}{\operatorname{GL}(n)}

	\newtheorem{theorem}{Theorem}[section]
	
	\newtheorem{lemma}[theorem]{Lemma}
	\newtheorem{corollary}[theorem]{Corollary}

\title{Weighted floating bodies and polytopal approximation}
\author{Florian Besau, Monika Ludwig and Elisabeth M.~Werner}
\date{}

\begin{document}
\maketitle
\begin{abstract}
Asymptotic results for weighted floating bodies are established and used to obtain new proofs for the existence of floating areas on the sphere and in hyperbolic space and to establish the existence of floating areas in Hilbert geometries.	Results on weighted best and random approximation and the new approach to floating areas are combined to derive new asymptotic approximation results 	on the sphere, in hyperbolic space and in Hilbert geometries.

\bigskip
\noindent
{\footnotesize 2000 AMS subject classification: Primary 52A38; Secondary 52A27, 52A55, 53C60, 60D05.}
\end{abstract}
\bigskip

Let $K$ be a convex body  (that is, compact convex set) in $\R^n$. For $\delta>0$, the floating body $K_\delta$ of $K$ is obtained by cutting off all caps that have volume less or equal to $\delta$.
Extending results for smooth bodies (cf.\ \cite{Leichtweiss:1986}), 
Sch\"utt and Werner \cite{SW:1990} showed for a general convex body $K$ that
\begin{align}\label{eqn:floatingbody}
	\lim_{\delta\to 0} \big(\!\Vol(K)-\Vol(K_\delta)\big) \delta^{-\frac{2}{n+1}} 
		&= \alpha_n \int_{\bd K} H_{n-1}(K,x)^{\frac{1}{n+1}}\, dx,
\end{align}
where $\alpha_n$ is an explicitly known positive constant (see Section \ref{weight_float}).
Here $\Vol$ is $n$-dimensional volume, $H_{n-1}(K,x)$ is the Gauss-Kronecker curvature at $x\in \bd K$ and integration is with respect  to the $(n-1)$-dimensional Hausdorff measure.
The integral on the right side is the {\em affine surface area} of $K$ 
(cf.\ \cite{Lutwak:1991, LR:1999} and \cite[Section 10.5]{Schneider:2014} for more information).

Affine surface area also determines the asymptotic behavior of random polytopes. Specifically,
choose $m$ points uniformly and independently in $K$ and denote their convex hull by $K_m$.
The \emph{random polytope} $K_m$ is easily seen to converge to $K$ in the sense that
$\E(\Vol(K)-\Vol(K_m))\to 0$ as $m\to\infty$, where  $\E$ denotes expectation.
The asymptotic behavior of $K_m$ has been studied
extensively since the 1960's, starting with the seminal results by
R\'enyi and Sulanke \cite{RS:1963, RS:1964} (cf.\ \cite{Reitzner:2010, Hug:2013}). 
Extending results of B\'ar\'any \cite{Barany:1992}, 
Sch\"utt \cite{Schuett:1994} was able to prove the analog 
to \eqnref{eqn:floatingbody} for the random polytope $K_m$ in a general convex body $K$,
\begin{align}\label{eqn:schuett}
	\lim_{m\to \infty} \E \big(\!\Vol(K) - \Vol(K_m)\big) \,{m^\frac{2}{n+1}}
		&= \beta_n \, \Vol(K)^\frac2{n+1}\,\int_{\bd K}  H_{n-1}(K,x) ^\frac{1}{n+1}\, dx,
\end{align}
where $\beta_n$ is an explicitly known positive constant (see Section \ref{random_polytope}).

\goodbreak
The aim of the paper is to extend (\ref{eqn:floatingbody}) and (\ref{eqn:schuett}) in a simple way
to convex bodies on the sphere, in hyperbolic space and in Hilbert geometries.
The approach is via {\em weighted floating bodies} and {\em weighted approximation} in Euclidean space and will also be applied to random approximation by circumscribed polytopes and to asymptotic best approximation.
On the sphere, the asymptotic behavior is described by the {\em spherical floating area}
recently introduced in \cite{BW:2015} and in hyperbolic space by the {\em hyperbolic floating area}
introduced in  \cite{BW:2016}.
In Hilbert geometries, we obtain floating areas that depend on the choice of volume and we establish a connection to centro-affine surface area.
\goodbreak

In the following section, a theorem for weighted floating bodies is stated and results on weighted approximation are collected. 
The results on polytopal random and best approximation and floating bodies on the sphere,
in hyperbolic space and in Hilbert geometries are established in Sections~\ref{sec:sphere},
\ref{sec:hyperbolic} and \ref{sec:hilbert}. The final section contains the proof for the theorem for weighted floating bodies.

\goodbreak
\section{Weighted floating bodies and polytopal approximation in $\R^n$}\label{sec:weighted}

Let $\cK(\rn)$ denote the set of convex bodies (that is, compact convex sets) in $\rn$ with non-empty interior.
For $K\in\cK(\rn)$ and $\phi,\psi:K\to(0,\infty)$ integrable, define,  for $A\subset \rn$  measurable, the measure $\Phi$ by $\Phi(A)=\int_A \phi$ and the measure $\Psi$ by $\Psi(A)=\int_A \psi$.  If $\int_\rn \phi=1$, then $\Phi$ is a probability measure and we write $\E_\Phi$ for the expectation with respect to $\Phi$.

\goodbreak
\subsection{Weighted floating bodies}\label{weight_float}
For $\delta >0$, the weighted floating body $K^\phi_\delta$ is the intersection of all
closed half-spaces
whose defining hyperplanes $H$ cut off  sets of $\Phi$-measure less than
or equal to $\delta$ from $K$, that is,
\begin{equation}\label{eqn:wfb}
	K^\phi_{\delta}=\bigcap \big\{ H^\ms : \Phi(K \cap H^\ps)\le \delta\big\},
\end{equation}
where $H^\pm$ are the closed half-spaces bounded by the hyperplane $H$.
For $\phi\equiv 1$, we obtain (convex) floating bodies, which were introduced (independently) in \cite{BL:1988, SW:1990} as a generalization of the classical {\em floating bodies} (see \cite[Chapter 10.6]{Schneider:2014} for more information). Weighted floating bodies were introduced in \cite{Werner:2002} and generalizations of   \eqnref{eqn:floatingbody} were established there.

The following result generalizes those results from volume to a general measure $\Psi$.
\begin{theorem}\label{thm:wfb}
	For $K\in\cK(\rn)$ and $\phi, \psi: K\to (0,\infty)$  continuous,
	\begin{align}\label{eqn:limit}
		\lim_{\delta \to 0} \frac{\Psi(K) - \Psi(K^\phi_\delta)}{\delta^\frac{2}{n+1}}
			&= \alpha_n \int_{\bd K}  H_{n-1}(K,x) ^\frac{1}{n+1} \phi(x)^{-\frac2{n+1}} \psi(x) \,dx,
	\end{align}
	where 
	\begin{align}\label{eqn:const2}
		\alpha_n:=\frac{1}{2} \bigg(\frac{n+1}{v_{n-1}}\bigg)^{\frac{2}{n+1}}
	\end{align}
and $v_{n-1}$ is the $(n-1)$-dimensional volume of the $(n-1)$-dimensional unit ball.
\end{theorem}

\noindent
The proof is given in Section \ref{sec:proof}.

\subsection{Random polytopes}\label{random_polytope}

For $K\in\cK(\R^n)$, let  $\phi: K\to (0,\infty)$ be a probability density and $K^\Phi_m$ the convex hull of $m$ independent random points chosen according to $\Phi$. The following generalization of (\ref{eqn:schuett}) was established by B\"or\"oczky, Fodor, and Hug \cite[Theorem~3.1]{BFH:2010}. 

\begin{theorem}[\! \cite{BFH:2010}]\label{thm:wr}
	Let $K\in\cK(\rn)$ and $ \psi: K\to(0,\infty)$ be continuous.
	If $\phi: K\to (0,\infty)$ is a continuous probability density and the random polytope $K^\Phi_m$ is the convex hull of $m$ independent random points chosen according to $\Phi$, then
	\begin{align}\label{eqn:randomapprox}
		\lim_{m\to \infty} \E_{\Phi}\big(\Psi(K) - \Psi(K^\Phi_m)\big) \,{m^\frac{2}{n+1}}
			&= \beta_n \, \int_{\bd K}  H_{n-1}(K,x) ^\frac{1}{n+1} \phi(x)^{-\frac2{n+1}} \psi(x) \, dx,
	\end{align}
	where 
	\begin{align}\label{eqn:const}
		\beta_n 
			:= \frac{(n^2+n+2)(n^2+1)}{2(n+3)\cdot (n+1)!}\, \Gamma\bigg(\frac{n^2+1}{n+1}\bigg)\,\bigg(\frac{n+1}{v_{n-1}}\bigg)^{\frac{2}{n+1}}.
	\end{align}
\end{theorem}

Efron showed that from the expected volume of a random polytope, the expected number of vertices $f_0(K_m)$ can be easily obtained. The same argument applies here and
\begin{align*}
	\E_{\Phi} f_0(K^\Phi_m) = m \big(1-\E_\Phi \Phi(K_{m-1}^\Phi)\big),
\end{align*}
(cf.\ \cite{Hug:2013}).
B\"or\"oczky, Fodor, and Hug \cite[Corollary 3.2]{BFH:2010} deduced the following result.

\begin{corollary}[\!\! \cite{BFH:2010}]\label{cor:wr}
	Let $K\in\cK(\rn)$.
	If $\phi: K\to (0,\infty)$ is a continuous probability density  and the random polytope $K^\Phi_m$ is the convex hull of $m$ independent random points chosen in $K$ according to the probability measure $\Phi$, then
	\begin{align*}
		\lim_{m\to \infty} \E_{\Phi}f_0(K^\Phi_m) \,m^{-\frac{n-1}{n+1}}
		= \beta_n \, \int_{\bd K}  H_{n-1}(K,x)^\frac{1}{n+1} \phi(x)^{\frac{n-1}{n+1}}\, dx,
	\end{align*}
	where $\beta_n$ is the constant defined in \eqnref{eqn:const}.
\end{corollary}

\goodbreak

\subsection{Random polyhedral sets}
Another model for random polytopes, that was also suggested by R\'enyi and Sulanke \cite{RS:1968}
and that can be considered as dual to the above, is the following:
Given a convex body $K$ in $\R^n$, choose $m$ random closed half-spaces that contain $K$ in a way that is described below
and denote their intersection by $K^m$. 
The \emph{random polyhedral set} $K^m$ may be unbounded and therefore one usually considers $K^m$
intersected with a bounded neighborhood of $K$.
The classical choice is the parallel body $K+\B^n$ of $K$, where $\B^n$ is the closed Euclidean unit ball,
that is, $K+\B^n$ is the set of all points of distance at most $1$ from $K$.

To describe our choice of random half-spaces, we first consider the set $\mathcal{H}$ of all closed half-spaces in $\R^n$. We parametrize  closed half-spaces $H^-(u,t)$ by its normal $u\in\S^{n-1}$ and the distance $t$ from the origin, i.e.,
\begin{align*}
	H^-(u,t) := \{ x\in\R^n : x\cdot u \leq t\}.
\end{align*}
The \emph{support function} $h_K$ of $K$ is defined, for $u\in\R^n$, by $h_K(u) = \max \{u\cdot x : x\in K\}$.
For $u\in\S^{n-1}$, the support function measures the signed distance between the origin and a hyperplane with outer normal $u$ that touches $K$ and the width of $K$ in direction $u$ is given by $h_K(u)+h_K(-u)$. The average width $W(K)$,  also known as \emph{mean width} of $K$,  is,
\begin{align}\label{mean_width}
	W(K) = \frac{1}{n v_n} \int_{\S^{n-1}} \big(h_K(u)+h_K(-u)\big)\, du
		 = \frac{2}{n v_n} \int_{\S^{n-1}} h_K(u)\, du.
\end{align}
On $\mathcal{H}$, there is a uniquely determined rigid motion invariant Borel measure $\mu$  such that 
\begin{align*}
	\mu\big(\{H^-\in \mathcal{H} : 0 < V_n(K\cap H^-)/V_n(K) < 1 \}\big) = W(K).
\end{align*}
For a Borel subset $A$ of $\mathcal{H}$, it is defined by
\begin{align*}
	\mu(A) = \frac{1}{n v_n} \int_{\S^{n-1}} \int_{\R} \mathbf{1}\big[H^-(u,t)\in A\big] \, dt \, du,
\end{align*}
where $\mathbf{1}[P]$ is the indicator function of the proposition $P$, that is, $\mathbf{1}[P]=1$ if $P$ holds and $\mathbf{1}[P]=0$ otherwise. For $K\in \cK(\rn)$, we consider the set of all half-spaces that contain $K$ and whose boundary hyperplanes meet $K+\B^n$, i.e.,
\begin{align*}
	\mathcal{H}_K=\big\{H^-(u,t) : u\in\S^{n-1}, h_K(u)\leq t \leq h_K(u)+1\big\}.
\end{align*}
This yields $\mu(\mathcal{H}_K) = 1$ and therefore the restriction $\mu_K$ of $\mu$ to $\mathcal{H}_K$ is a probability measure. Write $\E_{\mu_K}$ for the expectation with respect to $\mu_K$.
B\"or\"oczky, Fodor, and Hug \cite{BFH:2010} obtained the following result, which can be seen as dual to Theorem \ref{thm:wr} and Corollary \ref{cor:wr}.
\begin{theorem}[\!\! \cite{BFH:2010}]\label{thm:wdr}
	Let $K\in\cK(\rn)$. If the random polyhedral set $K^m$ is the intersection of $m$ independent random half-spaces chosen from $\mathcal{H}_K$ according to $\mu_K$, then
	\begin{align*}
		\lim_{m\to \infty} \E_{\mu_K}\Big(W\big(K^m\cap (K+\B^n)\big)-W(K)\Big)\, m^{\frac{2}{n+1}} 
			&= 2 \beta_n\, (nv_n)^{-\frac{n-1}{n+1}}  \int_{\bd K} H_{n-1}(K,x)^{\frac{n}{n+1}}\, dx,\\
	\intertext{and}
		\lim_{m\to\infty} \E_{\mu_K} f_{n-1}(K^m)\, m^{-\frac{n-1}{n+1}} 
			&= \beta_n\, (nv_n)^{-\frac{n-1}{n+1}}  \int_{\bd K} H_{n-1}(K,x)^{\frac{n}{n+1}}\, dx,
	\end{align*}
	where $f_{n-1}(K^m)$ is the number of facets of $K^m$ and  $\beta_n$ is the constant from \eqnref{eqn:const}.
\end{theorem}

\goodbreak
\subsection{Weighted best approximation}
Problems of asymptotic best approximation have been extensively studied since the 1940's (cf.\ \cite{Gruber:1993a}). We restrict our attention to two problems and just remark that further notions of distance and approximation by inscribed and circumscribed polytopes with a given number of faces have also been studied  (cf.\ \cite{Gruber:1993a}).
For $K,P\subset \rn$, write $K\triangle P$ for the symmetric difference of $K$ and $P$. Set 
\begin{align*}
	\dist_\Psi\!\big(K,\cP_m\big)
		&=\inf\big\{\Psi(K\triangle P): P  \text{ polytope with at most $m$ vertices}\big\}\\
\intertext{and}
	\dist_\Psi\!\big(K,\cP_{(m)}\big)
		&=\inf\big\{\Psi(K\triangle P): P  \text{ polytope with at most $m$ facets}\big\}.
\end{align*}
Extending results by L. Fejes T\'oth \cite{FejesToth:1948} and Gruber \cite{Gruber:1993b}, the following asymptotic result was established in \cite{Ludwig:1998} for convex bodies with positive curvature and in \cite{Boroczky:2000} the curvature condition was dropped.

\begin{theorem}[\!\!\cite{Ludwig:1998, Boroczky:2000}]\label{thm:wb}
	For  $K\in\cK(\rn)$ with $C^2$ boundary and $\psi:K\to(0,\infty)$ continuous,
	\begin{align}\label{eqn:bestapprox}
		\lim_{m\to \infty} \dist_\Psi\!\big(K, \cP_m\big) \,{m^\frac{2}{n-1}}
		&=\frac12 \ldel_{n-1} \bigg(\int_{\bd K}  H_{n-1}(K,x) ^\frac{1}{n+1}  \psi(x) ^{\frac{n-1}{n+1}} \, dx\bigg)^{\frac{n+1}{n-1}},\\
	\intertext{and}\notag
		\lim_{m\to \infty} \dist_\Psi\!\big(K, \cP_{(m)}\big) \,{m^\frac{2}{n-1}}
		&=\frac12 \ldiv_{n-1} \bigg(\int_{\bd K}  H_{n-1}(K,x) ^\frac{1}{n+1}  \psi(x) ^{\frac{n-1}{n+1}} \, dx\bigg)^{\frac{n+1}{n-1}},
	\end{align}
	where $\ldel_{n-1}$ and $\ldiv_{n-1}$ are positive constants.
\end{theorem}

\noindent
The exact values of  $\ldel_{n-1}$ and $\ldiv_{n-1}$ are only known for $n=2$ and $n=3$ (see \cite{BL:1999}). Weighted best approximation was first considered by Glasauer (see \cite{GG:1997}).

\goodbreak
\section{Spherical space}\label{sec:sphere}

Let $\S^n$ denote the unit sphere in $\R^{n+1}$. A set $K\subset \S^n$ is a proper convex body, if it is closed, contained in an open hemisphere and its positive hull $\mathrm{pos}\, K=\{\lambda x: x\in K, \lambda \geq 0\}$ is a convex set in $\R^{n+1}$.
Let $\cK(\S^n)$ denote the set of proper convex bodies in $\S^n$ with non-empty interior.
A hypersphere in $\S^n$ is a set $H=\{x\in\S^n: x\cdot e=0\}$ with $e\in\S^n$, where $\lq\lq\cdot"$ is the inner product in $\R^{n+1}$. Let $H^\pms$ be the closed hemispheres bounded by $H$.
For $\delta >0$, the spherical floating body $K_\delta$ was introduced in \cite{BW:2015} by
\begin{equation}\label{eqn:sfb}
	K_{\delta}=\bigcap \big\{ H^\ms : \vol_n(K\cap H^\ps)\le \delta\big\},
\end{equation}
where $\vol_n$ is spherical volume, that is, the $n$-dimensional Hausdorff measure on $\S^n$. 

Without loss of generality, we may restrict our attention to convex bodies contained in the hemisphere $\S^n_\ps=\{x\in\S^n: x\cdot e_{n+1}>0\}$, where $e_{n+1}$ is a vector of an orthonormal basis of $\R^{n+1}$. The gnomonic (or central) projection $g:\S^n_\ps\to \R^n$ is defined by 
\begin{align*}
	g(x) = \frac x{x\cdot e_{n+1}} -e_{n+1},
\end{align*}
where we identify $\R^n$ with $\{x\in\R^{n+1}: x\cdot e_{n+1}=0\}$ (cf.~\cite[Sec.~4]{BS:2016}). We write $\bar x=g(x)$ and $\bar K =g(K)$.
Note that $g^{-1}: \R^n \to \S^n$ maps the point $\bar x$ to $(1+\|\bar x\|^2)^{-1/2} (\bar x+e_{n+1})$ and has therefore the Jacobian 
$(1+\|\bar x\|^2)^{-(n+1)/2}$ (cf.\ \cite[Proposition 4.2]{BW:2015}). Thus 
the pushforward of  $\vol_n$ under $g$ is the measure $\Psi_n$ with density $\psi_n(\bar x)= (1+\|\bar x\|^2)^{-(n+1)/2}$. For the spherical Gauss-Kronecker curvature, we have
\begin{equation*}\label{sgk}
	\Hs(K, x)
		= H_{n-1}(\bar K, \bar x) \bigg(\frac{1+\|\bar x\|^2}{1+(\bar x\cdot n_{\bar K}(\bar x))^2}\bigg)^{\frac{n+1}2}
\end{equation*}
(cf.\ \cite[Lemma 4.4]{BW:2015}), where $n_{\bar K}(\bar x)$ is the outer unit normal vector to $\bar K$ at $\bar x$,
and consequently
\begin{equation}\label{eqn:area}
	\int_{\bd K}  \Hs(K,x) ^\frac{1}{n+1} \ dx 
		= \int_{\bd \bar K}  H_{n-1}(\bar K,\bar x) ^\frac{1}{n+1} (1+\|\bar x\|^2)^{-\frac{n-1}2} \, d\bar x
\end{equation}
(cf.\ \cite[p.\ 897]{BW:2015}).
These transformation rules allow us to translate the results from Section~\ref{sec:weighted} to spherical space.

The following result is a corollary to Theorem \ref{thm:wfb} and was first established in \cite{BW:2015}.

\begin{theorem}[\!\! \cite{BW:2015}]\label{thms:wfb}
	For $K\in\cK(\S^n)$, 
	\begin{align*}
		\lim_{\delta \to 0} \frac{\vol_n(K) - \vol_n(K_\delta)}{\delta^\frac{2}{n+1}}
		= \alpha_n \int_{\bd K}  \Hs(K,x) ^\frac{1}{n+1}  dx,
	\end{align*}
	where $\alpha_n$ is the constant from \eqnref{eqn:const2}.
\end{theorem}

\begin{proof}
	Since $g(K_\delta)= g(K)^{\psi_n}_\delta$, we have
	\begin{align*}
		\vol_n(K)- \vol_n(K_\delta)=\int_{g(K)\backslash g(K)^{\psi_n}_\delta} \psi_n.
	\end{align*}
	Hence Theorem \ref{thm:wfb} with $\phi=\psi=\psi_n$ shows that
	\begin{align*}
		\lim_{\delta \to 0} \frac{\vol_n(K) - \vol_n(K_\delta)} {\delta^\frac{2}{n+1}}
		= \alpha_n \int_{\bd \bar K}  H_{n-1}(\bar K,\bar x) ^\frac{1}{n+1} (1+\|\bar x\|^2)^{-\frac{n-1}2} \ d\bar x.
	\end{align*}
	By \eqnref{eqn:area}, this completes the proof.
\end{proof}

Next, we consider random polytopes that are the spherical convex hull of points chosen uniformly according to $\vol_n$  in $K\in\cK(\S^n)$. In the following, the expectation $\E_K$ is with respect to the probability density $\vol_n/\vol_n(K)$. 

\begin{theorem}\label{thms:wr}
	Let $K\in\cK(\S^n)$. 
	If $K_m$ is the spherical convex hull of $m$ random points chosen uniformly in $K$, then
	\begin{align*}
		\lim_{m\to \infty} \E_K\big(\!\vol_n(K) - \vol_n(K_m)\big) \,{m^\frac{2}{n+1}}
			&= \beta_n \, \vol_n(K)^{\frac{2}{n+1}} \int_{\bd K}  \Hs(K,x)^\frac{1}{n+1} \, dx,
	\end{align*}
	where $\beta_n$ is the constant from \eqnref{eqn:const}.
\end{theorem}

\begin{proof}
	Set $\Phi_n=\Psi_n/\Psi_n(g(K))$. Since $g(K_m)=g(K)^{\Phi_n}_m$, we have
	\begin{align*}
		\E_K\big(\!\vol_n(K) - \vol_n(K_m)\big )
			& = \E_{\Phi_n} \big( \Psi_n(g(K)- \Psi_n(g(K)^{\Phi_n}_m)\big).
	\end{align*}
	Thus the statement follows from Theorem \ref{thm:wr} with $\psi=\psi_n$ and \eqnref{eqn:area}.
\end{proof}

\noindent Theorem \ref{thms:wr} complements a recent result by B\'ar\'any, Hug, Reitzner and Schneider \cite{BHRS:2016} for random polytopes in hemispheres.

\goodbreak

As a consequence of Corollary \ref{cor:wr}, we obtain the following result.

\begin{corollary}\label{cors:wr}
	Let $K\in\cK(\S^n)$.
	If $K_m$ is the spherical convex hull of $m$ random points chosen uniformly in $K$, then
	\begin{align*}
		\lim_{m\to \infty} \E_K f_0(K_m) \,m^{-\frac{n-1}{n+1}}
			= \beta_n\, \vol_n(K)^{-\frac{n-1}{n+1}} \int_{\bd K}  H^\sn_{n-1}(K,x) ^\frac{1}{n+1} \, dx,
	\end{align*}
	where $\beta_n$ is the constant from \eqnref{eqn:const}.
\end{corollary}

Finally, we consider best approximation.
Let 
\begin{align*}
	\dist_n\!\big(K,\cP^\sn_m\big)&=\inf\big\{\!\vol_n(K\triangle P): P  \text{ spherical polytope with at most $m$ vertices}\big\},\\
\intertext{and}
	\dist_n\!\big(K,\cP^\sn_{(m)}\big)&=\inf\big\{\!\vol_n(K\triangle P): P  \text{ spherical polytope with at most $m$ facets}\big\}.
\end{align*}
We obtain the following result. 

\begin{theorem}\label{thms:wb}
	For $K\in\cK(\S^n)$ with $C^2$ boundary,
	\begin{align*}
		\lim_{m\to \infty} \dist_n\!\big(K, \cP^\sn_m\big) \,{m^\frac{2}{n-1}}
		&= \frac12 \ldel_{n-1} \bigg(\int_{\bd K}  \Hs(K,x)^\frac{1}{n+1}   \ dx\bigg)^{\frac{n+1}{n-1}},\\
	\intertext{and}
		\lim_{m\to \infty} \dist_n\!\big(K, \cP^\sn_{(m)}\big) \,{m^\frac{2}{n-1}}
		&= \frac12 \ldiv_{n-1} \bigg(\int_{\bd K}  \Hs(K,x)^\frac{1}{n+1}   \ dx\bigg)^{\frac{n+1}{n-1}},
	\end{align*}
	where $\ldel_{n-1}$ and $\ldiv_{n-1}$ are the constants from Theorem \ref{thm:wb}.
\end{theorem}

\begin{proof}
	Note that
	$\dist_n(K,\cP^\sn_m)= \dist_{\Psi_n}(g(K), \cP_m)$ and $\dist_n(K,\cP^\sn_{(m)})= \dist_{\Psi_n}(g(K), \cP_{(m)})$. Thus the statement follows directly from Theorem~\ref{thm:wb} with $\psi=\psi_n$ and \eqnref{eqn:area}.
\end{proof}

\subsection{Duality principle}

Let $K$ be a proper spherical convex body. Instead of random polytopes $K_m$ contained in $K$ we now consider random polytopes $K^m$ containing $K$. The space of closed hemispheres $\mathcal{H}$ of $\S^n$ has a uniquely determined rotation invariant probability measure $\mu$. 
For each point $x\in\S^n$ there is a uniquely determined hemisphere $H^-(x)=\{y\in\S^n:x\cdot y\leq 0\}$ and for a Borel subset $A$ of $\mathcal{H}$ we have
\begin{align*}
	\mu(A) = \frac{1}{\vol_n(\S^n)} \int_{\S^n} \mathbf{1}\big[H^-(x) \in A\big] \, dx.
\end{align*}
A random polytope $K^m$ is obtained as intersection of $m$ closed hemispheres chosen from $\mathcal{H}_K:=\{H^-\in \mathcal{H}: K\subseteq H^-\}$ independently and according to $\mu_K:=\mu/\mu(\mathcal{H}_K)$.

For $K\in\cK(\S^n)$, define the \emph{polar body} $K^\circ$ by
\begin{align*}
	K^\circ = \{y\in\S^n : x\cdot y \leq 0 \text{ for all $x\in K$}\} =\bigcap_{x\in K} H^-(x)
\end{align*}
(cf.~\cite[Sec.~6.5]{SW:2008}).
Since $K^{\circ\circ} = K$,  a hemisphere $H^-(y)$ contains $K$ if and only if $y\in K^\circ$. Thus we have $\mathcal{H}_K = \{H^-(y):y\in K^\circ\}$ and $\mu(\mathcal{H}_K) = \vol_n(K^\circ)$.

Let $K^m$ be the intersection of $m$ randomly chosen closed hemispheres in $\mathcal{H}_K$, that is, there are $x_i \in K^\circ$, $i=1,\ldots,m$, such that $K^m = \bigcap_{i=1}^m H^-(x_i)$. We have
\begin{align*}
	K^m = \big(\mathrm{conv}\{x_1,\ldots,x_m\}\big)^\circ = \big(K^\circ_m\big)^\circ,
\end{align*}
where $K^\circ_m := (K^\circ)_m$. This means, that the polar of a random polytope that contains $K$ is a polytope inside $K^\circ$. In this way we can transfer results about $K^{\circ}_m$ to $(K^m)^{\circ}$.

\begin{theorem}\label{thm:dual}
	If $\,\mathcal{F}$ be a non-negative measurable functional on spherical convex polyhedral sets, then
	\begin{align*}
		\E_{\mu_K} \mathcal{F}(K^m) = \E_{K^\circ} \mathcal{F}\big(\left(K^\circ_m\right)^\circ\big).
	\end{align*}
\end{theorem}
\noindent
In the Euclidean setting a similar results was obtained in \cite[Prop.~5.1]{BFH:2010}.

As an application of this theorem we consider the \emph{spherical mean width} $U_1(K)$ of a spherical convex body $K$, which is defined by
\begin{align*}
	U_1(K) = \frac{1}{2} \int_{G(n+1,n)} \chi(K\cap H)\, d\nu(H),
\end{align*}
where $\chi$ is the Euler characteristic, $G(n+1,n)$ is the Grassmannian of all $n$-dimensional linear subspaces in $\R^{n+1}$ and $\nu$ denotes the invariant probability measure on $G(n+1,n)$. The probability that a random hypersphere hits $K$ is equal to $2U_1(K)$. The name {\em spherical mean width} corresponds to the Euclidean notion of mean width $W(\bar{K})$ for $\bar{K}\in \cK(\R^n)$, which can be defined as the probability of a random affine hyperplane hitting $\bar{K}$. Equivalently, $W(\bar{K})$ is given by (\ref{mean_width}), which, however, does not have a natural analog in the spherical setting.

\begin{corollary}\label{cors:meanwidth}
	Let $K\in \cK(\S^n)$. If $K^m$ is the intersection of $m$ random hemispheres containing $K$ and chosen uniformly according $\mu_K$, then
	\begin{align*}
		\lim_{m\to \infty} \E_{\mu_K}\big(U_1(K^m)-U_1(K)\big)\, m^{\frac{2}{n+1}} 
			&= \frac{\beta_n}{\vol_n(\S^n)}\, \vol_n(K^\circ)^{\frac{2}{n+1}} \int_{\bd K} H_{n-1}^{\S^n}(K,x)^{\frac{n}{n+1}}\, dx,\\
	\intertext{and}
		\lim_{m\to \infty} \E_{\mu_K} f_{n-1}(K^m)\, m^{-\frac{n-1}{n+1}} 
			&= \beta_n\, \vol_n(K^\circ)^{-\frac{n-1}{n+1}} \int_{\bd K} H_{n-1}^{\S^n}(K,x)^{\frac{n}{n+1}} \, dx,
	\end{align*}
	where $\beta_n$ is the constant from \eqnref{eqn:const}.
\end{corollary}
\begin{proof}
	By \cite[Eqn.~(20)]{GHS:2002}, we have
	\begin{align*}
		U_1(K) = \frac{1}{2} - \frac{\vol_n(K^\circ)}{\vol_n(\S^n)}.
	\end{align*}
	Also, the facets of $K^m$ correspond to the vertices of $(K^m)^\circ = K^\circ_m$. Thus $f_{n-1}(K^m) = f_0(K_m^\circ)$.
	Hence, by Theorem \ref{thm:dual}, we find
	\begin{align*}
		\E_{\mu_K}\big(U_1(K^m)-U_1(K)\big) 
			= \frac{\E_{K^\circ}\big(\!\vol_n(K^\circ) - \vol_n(K^\circ_m)\big)}{\vol_n(\S^n)},
	\end{align*}
	and
	\begin{align*}
		\E_{\mu_K} f_{d-1}(K^m) = \E_{K^\circ} f_0(K_m^\circ).
	\end{align*}
	Applying Theorem \ref{thms:wr} and Corollary \ref{cors:wr} on $K^\circ$ we obtain
	\begin{align*}
		\lim_{m\to \infty} \E_{\mu_K}\big(U_1(K^m)-U_1(K)\big)\, m^{\frac{2}{n+1}} 
			&= \frac{\beta_n}{\vol_n(\S^n)}\, \vol_n(K^\circ)^{\frac{2}{n+1}} \int_{\bd K^\circ} H_{n-1}^{\S^n}(K^\circ,x)^{\frac{1}{n+1}}\, dx,
	\end{align*}
	and 
	\begin{align*}
		\lim_{m\to \infty} \E_{\mu_K} f_{n-1}(K^m)\, m^{-\frac{n-1}{n+1}} 
			= \beta_n\, \vol_n(K^\circ)^{-\frac{n-1}{n+1}} \int_{\bd K^\circ} H_{n-1}^{\S^n}(K^\circ,x)^{\frac{1}{n+1}} \, dx,
	\end{align*}
	By \cite[Thm.~7.4]{BW:2015}, we have
	\begin{align*}
		\int_{\bd K^\circ} H_{n-1}^{\S^n}(K^\circ,x)^{\frac{1}{n+1}}\, dx = 
		\int_{\bd K} H_{n-1}^{\S^n}(K,x)^{\frac{n}{n+1}}\, dx,
	\end{align*}
	which concludes the proof.
\end{proof}

\goodbreak
\section{Hyperbolic space}\label{sec:hyperbolic}

Let $\R^{n,1}$ denote the Lorentz-Minkowski space of dimension $n+1$, that is,  $\R^{n+1}$ with the indefinite inner product ``$\circ$'' defined by
\begin{align*}
	x\circ x 
		= x_1^2+\dots+ x_n^2 - x_{n+1}^2.
\end{align*}
Then the hyperboloid model of hyperbolic space is given by
\begin{align*}
	\H^n
		=\big\{x\in\R^{n,1} : x\circ x = -1 \text{ and } x_{n+1} > 0 \big\}.
\end{align*}
The hyperbolic distance $d_H$ between two points $x,y\in\H^n$ is determined by $\cosh\, d_H(x,y) = -x\circ y$.
A set $K\subset \H^n$ is a convex body, if it is compact and the positive hull is a convex set in $\R^{n+1}$. Let $\cK(\H^n)$ denote the set of convex bodies in $\H^n$ with non-empty interior. 
For a hyperplane $H$ let $H^{\pms}$ be the closed half-spaces bounded by $H$. For $\delta>0$, the hyperbolic floating body $K_\delta$ was introduced in \cite{BW:2016} by
\begin{align*}
	K_{\delta} = \bigcap \big\{H^\ms:\vol_n(K\cap H^\ps)\le \delta\big\},
\end{align*}
where $\vol_n$ is the hyperbolic volume on $\H^n$.

We fix a Lorentz-orthonormal basis $e_1,e_2,\ldots, e_{n+1}$ in $\R^{n,1}$ such that $e_{n+1}$ is in $\H^n$. The gnomonic (or central) projection $g\colon \H^n\to \R^n$ is defined by
\begin{align*}
	g(x) = \frac{x}{x\circ e_{n+1}} + e_{n+1},
\end{align*}
where we identify $\R^n$ with $\{x\in\R^{n,1}:x\circ e_{n+1} = 0\}$. We write $\bar{x} = g(x)$ and $\bar{K}=g(K)$. 
Since 
\begin{align*}
	\|\bar{x}\|^2 = 1-(x\circ e_{n+1})^{-2} = \tanh^2\, d_H(x,e_{n+1}),
\end{align*}
we have $\|\bar{x}\|\in [0,1)$. 
Therefore the gnomonic projection maps $\H^n$ into the open unit ball $\Int \B^n\subset \R^n$.
Note that $g^{-1}\colon \Int \B^n\to \H^n$ maps the point $\bar{x}$ to $(1-\|\bar{x}\|^2)^{-1/2}(\bar{x}+e_{n+1})$. The gnomonic projection is an isometry between the hyperboloid model $\H^n$ and the projective model (or Beltrami--Cayley--Klein model) $\Int \B^n$. Thus the pushforward of $\vol_n$ under $g$ is the measure $\Psi_n$ with density $\psi_n(\bar{x})  = (1-\|\bar{x}\|^2)^{-(n+1)/2}$. For the hyperbolic Gauss--Kronecker curvature, we have 
\begin{align*}
	\Hh(K,x) = H_{n-1}(\bar{K},\bar{x})\left(\frac{1-\|\bar{x}\|^2}{1-(\bar{x}\cdot n_{\bar{K}}(\bar{x}))^2}\right)^{\frac{n+1}{2}}
\end{align*}
(cf.\ \cite[Cor.\ 3.16]{BW:2016}), and furthermore
\begin{equation}\label{eqn:areah}
	\int_{\bd K} \Hh(K,x)^{\frac{1}{n+1}}\, dx = \int_{\bd \bar{K}} H_{n-1}(\bar{K},\bar{x})^{\frac{1}{n+1}} (1-\|\bar{x}\|^2)^{-\frac{n-1}{2}}\, d\bar{x}
\end{equation}
(cf.\ \cite[(3.12)]{BW:2016}). So again, these transformation rules allow us to translate the results from Section~1 to hyperbolic space. The proofs are identical to those in spherical space (just replace (\ref{eqn:area}) by (\ref{eqn:areah})) and are therefore omitted.

As a corollary to Theorem \ref{thm:wfb} we obtain the existence of floating area for hyperbolic space, which was originally established in \cite{BW:2016}.
\begin{theorem}[\!\! \cite{BW:2016}]\label{thmh:wfb}
	For $K\in\cK(\H^n)$, 
	\begin{align*}
		\lim_{\delta \to 0} \frac{\vol_n(K) - \vol_n(K_\delta)} {\delta^\frac{2}{n+1}}
		= \alpha_n \int_{\bd K}  \Hh(K,x) ^\frac{1}{n+1}  dx,
	\end{align*}
	where $\alpha_n$ is defined in Theorem \ref{thm:wfb}.
\end{theorem}

Next, we consider random polytopes that are the hyperbolic convex hull of points chosen uniformly according to $\vol_n$  in $K\in\cK(\H^n)$. In the following, the expectation $\E_K$ is with respect to the density $\vol_n/\vol_n(K)$. 

\begin{theorem}\label{thmh:wr}
	Let $K\in\cK(\H^n)$. 
	If $K_m$ is the hyperbolic convex hull of $m$ random points chosen uniformly in $K$, then
	\begin{align*}
		\lim_{m\to \infty} \E_K\big(\!\vol_n(K) - \vol_n(K_m)\big) \,{m^\frac{2}{n+1}}
			&= \beta_n \, \vol_n(K)^{\frac{2}{n+1}} \int_{\bd K}  \Hh(K,x)^\frac{1}{n+1} \ dx,
	\end{align*}
	where $\beta_n$ is the constant from \eqnref{eqn:const}.
\end{theorem}

As a consequence, we obtain the following result.

\begin{corollary}\label{corh:wr}
	Let $K\in\cK(\H^n)$.
	If $K_m$ is the hyperbolic convex hull of $m$ random points chosen uniformly in $K$, then
	\begin{align*}
		\lim_{m\to \infty} \E_K f_0(K_m) \,m^{-\frac{n-1}{n+1}}
			&= \beta_n \, \vol_n(K)^{-\frac{n-1}{n+1}} \int_{\bd K}  H^\hn_{n-1}(K,x) ^\frac{1}{n+1} \, dx,
	\end{align*}
	where $\beta_n$ is the constant from \eqnref{eqn:const}.
\end{corollary}

Finally, we consider best approximation.
Let 
\begin{align*}
	\dist_n\big(K,\cP^\hn_m\big)&=\inf\big\{\!\vol_n(K\triangle P): P  \text{ hyperbolic polytope with at most $m$ vertices}\big\},\\
\intertext{and}
	\dist_n\big(K,\cP^\hn_{(m)}\big)&=\inf\big\{\!\vol_n(K\triangle P): P  \text{ hyperbolic polytope with at most $m$ facets}\big\}.
\end{align*}
We obtain the following result. 

\begin{theorem}\label{thmh:wb}
	For $K\in\cK(\S^n)$ with $C^2$ boundary,
	\begin{align*}
		\lim_{m\to \infty} \dist_n\big(K, \cP^\hn_m\big) \,{m^\frac{2}{n-1}}
			&= \frac12 \ldel_{n-1} \bigg(\int_{\bd K}  \Hh(K,x)^\frac{1}{n+1}   \ dx\bigg)^{\frac{n+1}{n-1}},\\
	\intertext{and}
		\lim_{m\to \infty} \dist_n\big(K, \cP^\hn_{(m)}\big) \,{m^\frac{2}{n-1}}
			&= \frac12 \ldiv_{n-1} \bigg(\int_{\bd K}  \Hh(K,x)^\frac{1}{n+1}   \ dx\bigg)^{\frac{n+1}{n-1}},
	\end{align*}
	where $\ldel_{n-1}$ and $\ldiv_{n-1}$ are the constants from Theorem \ref{thm:wb}.
\end{theorem}

\goodbreak
\section{Hilbert geometries}\label{sec:hilbert}

Hilbert's Fourth Problem asks for a characterization of metric geometries
whose geodesics are straight lines. Hilbert constructed a special class of
examples, now called Hilbert geometries (see \cite[Ch.~15]{HandbookHilbert:2014} for more information).
A Hilbert geometry $(C,d_C)$ is defined on the interior of a convex body $C\in \cK(\rn)$
in the following way:
For distinct points $x, y \in \Int C$, the line passing through $x$ and $y$ meets $\bd C$  at two points $p$ and $q$, say, such that one has $p, x, y, q$ in that order on the line. Define the Hilbert distance of $x$ and $y$ by
\begin{align*}
	d_C(x, y) = \frac12 \log [p,x, y, q],
\end{align*}
where $[p,x,y, q]$ is the cross ratio of $p,x, y, q$, that is,
\begin{align*}
	[p,x, y, q] = \frac{\| y - p\|}{\|x - p\|} \, \frac{\| x - q\|}{\|y - q\|}.
\end{align*}
Note that the invariance of the cross ratio by projective maps implies the
projective invariance of $d_C$.
Unbounded closed convex sets with nonempty interiors and not containing a
straight line are projectively equivalent to convex bodies.  Hence the definition of Hilbert geometry naturally extends to the interiors of such convex sets.  If $C$ is an ellipsoid, then the Hilbert geometry on $\Int C$ is isometric to hyperbolic space. 

Straight lines are geodesics in a Hilbert geometry $(C,d_C)$ and if $C$ is strictly convex,
then the affine segment between to distinct points is the unique geodesic joining them (see e.g. \cite[p.~60]{HandbookHilbert:2014}). 
Hence, if $C$ is strictly convex, then hyperplanes are the totally geodesic submanifolds of co-dimension $1$. 
A convex body $K\in\mathcal{K}(\rn)$ that is contained in $\Int C$
is therefore also a convex body of the Hilbert geometry $(C,d_C)$ and polytopes are an intrinsic notion of $(C, d_C)$.
Thus we may consider polytopal approximation in 
a Hilbert geometry $(C,d_C)$ for a strictly convex body $C$. In the following $\mathcal{K}(C)$ denotes the space of convex bodies $K\subset \Int C$.

The Hilbert metric $d_C$ is induced by a weak Finsler structure in the following way: 
For $x\in \Int C$ define a (weak) Minkowski norm $\|.\|_x$ by
\begin{align*}
	\|v\|_{x}= \frac12 \Big( \frac1{t^\ps} +\frac1{t^\mss}\Big),
\end{align*}
for $v\in\R^n$, where $t^\pms$ is determined by $x\pm t^\pms v\in \bd C$. 
If we identify $\R^n$ with the tangent space $T_x\,\R^n$,
then $\|\cdot\|_x$ defines a Minkowski norm on $T_x\,\R^n$ for every $x\in\Int C$. 
The map $F_C:x \mapsto \|\cdot\|_x$ defines a (weak) Finsler structure on $\Int C$. The length of a $C^1$ curve $\gamma : [a,b]\to \Int C$ is defined by
\begin{align*}
	\ell(\gamma) = \int_a^b \|\dot{\gamma}(t)\|_{\gamma(t)}\, dt, 
\end{align*}
and the Hilbert metric between two distinct points $x,y\in\Int C$ is just the minimal length of a $C^1$ curve joining them. 
In particular, if $C$ is $C^2_+$, that is, the boundary of $C$ is a $C^2$ manifold 
with positive curvature, then $(\Int C, F_C)$ 
defines a Finsler manifold in the classical sense.

The unit ball of the Minkowski norm $\|\cdot\|_x$ is 
$I^C_{x}=\{v\in \rn: \| v \|_{x}\le 1\}$. 
Recall that the polar body $K^*$ of a convex body $K$ is defined by 
$K^*=\{y\in\R^n: x\cdot y \leq 1 \text{ for all $x\in K$}\}$ 
and the difference body $D\,K$ is defined by $D(K)=\frac{1}{2}(K-K) =\frac12\{x-y: x,y\in K\}$. 
For a fixed $x\in \Int C$ we find
\begin{align*}
	\|v\|_{x} = h\big(D (C-x)^*,v\big) \text{ and }
	I^C_{x} = \big(D (C-x)^*\big)^*.
\end{align*}
Hence $I^C_x$ is the \emph{harmonic symmetrization} of $C$ in $x$ (see \cite{PT:2009}). 
\goodbreak

There are several good choices for volume $\vol_C$ in $(C,d_C)$ 
which give a projective invariant notion of volume; 
for example, the Busemann volume or the Holmes-Thompson volume of the associated Finsler manifold. 
The Busemann volume is the $n$-dimensional Hausdorff volume of the metric space $(C,d_C)$. 
Its density function with respect to Lebesgue measure $\lambda_n$ is given by
$v_n/\lambda_n(I^C_x)$. The Holmes-Thompson volume has density 
$\lambda_n((I^C_x)^*)/v_n$.
Both, the Busemann and the Holmes-Thompson volume, 
have the property that the density $\sigma_C$ is non-negative and continuous. 
This allows us to directly apply the results from Section~1 
to Hilbert geometries with these volume densities. 

First, we consider random polytopes that are the convex hull of 
points chosen uniformly according to $\vol_C$  in $K\in\cK(C)$.
In the following, the expectation $\E_K$ is with respect to the density $\vol_C/\vol_C(K)$.

\begin{theorem}\label{thmhg:wr}
	Let $K\in\cK(C)$. 
	If $K_m$ is the convex hull of $m$ random points chosen uniformly in $K$ with respect to $\vol_C$, then
	\begin{align*}
		\lim_{m\to \infty} \E_K\big(\!\vol_C(K) - \vol_C(K_m)\big) \,{m^\frac{2}{n+1}}
			&= \beta_n \, \vol_C(K)^{\frac{2}{n+1}} \int_{\bd K}  H_{n-1}(K,x) ^\frac{1}{n+1} \sigma_C(x)^{\frac{n-1}{n+1}}\, dx,
	\end{align*}
	where $\beta_n$ is the constant from \eqnref{eqn:const}.
\end{theorem}

As a consequence of Corollary \ref{cor:wr}, we obtain the following result.

\begin{corollary}\label{corhg:wr}
	Let $K\in\cK(C)$.
	If $K_m$ is the convex hull of $m$ random points chosen uniformly in $K$ with respect to $\vol_C$, then
	\begin{align*}
		\lim_{m\to \infty} \E_K f_0(K_m) \,m^{-\frac{n-1}{n+1}}
			&= \beta_n \, \vol_C(K)^{-\frac{n-1}{n+1}} \int_{\bd K}  H_{n-1}(K,x) ^\frac{1}{n+1} \sigma_C(x)^{\frac{n-1}{n+1}}\, dx,
	\end{align*}
	where $\beta_n$ is the constant from \eqnref{eqn:const}.
\end{corollary}

Next, we consider best approximation.
Let 
\begin{align*}
	\dist_C\big(K,\cP^C_m\big)
		&=\inf\big\{\!\vol_C(K\triangle P): P\subset \Int C  \text{ polytope with at most $m$ vertices}\big\},\\
\intertext{and}
	\dist_C\big(K,\cP^C_{(m)}\big)
		&=\inf\big\{\!\vol_C(K\triangle P): P\subset \Int C  \text{ polytope with at most $m$ facets}\big\}.
\end{align*}
We obtain the following result. 

\begin{theorem}\label{thmhg:wb}
	For $K\in\cK(C)$ with $C^2$ boundary,
	\begin{align*}
		\lim_{m\to \infty} \dist_C\big(K, \cP^C_m\big) \,{m^\frac{2}{n-1}}
			&= \frac12 \ldel_{n-1} \bigg(\int_{\bd K}  H_{n-1}(K,x) ^\frac{1}{n+1}  \sigma_C(x) ^{\frac{n-1}{n+1}} \, dx\bigg)^{\frac{n+1}{n-1}},\\
	\intertext{and}
		\lim_{m\to \infty} \dist_C\big(K, \cP^C_{(m)}\big) \,{m^\frac{2}{n-1}}
			&= \frac12 \ldiv_{n-1} \bigg(\int_{\bd K}  H_{n-1}(K,x) ^\frac{1}{n+1}  \sigma_C(x) ^{\frac{n-1}{n+1}} \, dx\bigg)^{\frac{n+1}{n-1}},
	\end{align*}
	where $\ldel_{n-1}$ and $\ldiv_{n-1}$ are  the constants from Theorem \ref{thm:wb}.
\end{theorem}

Finally, we obtain the following result for the weighted floating body $K^{\sigma_C}_\delta$.

\begin{theorem}\label{thmhg:wfb}
	For $K\in\cK(C)$, 
	\begin{align*}
		\lim_{\delta \to 0} \frac{\vol_C(K) - \vol_C(K^{\sigma_C}_\delta)} {\delta^\frac{2}{n+1}}
			&= \alpha_n \int_{\bd K}  H_{n-1}(K,x) ^\frac{1}{n+1} \sigma_C(x)^{\frac{n-1}{n+1}} \,dx,
	\end{align*}
	where $\alpha_n$ is defined in Theorem \ref{thm:wfb}.
\end{theorem}

Note that the floating area
\begin{align*}
	\Omega_C(K)=\int_{\bd K}  H_{n-1}(K,x) ^\frac{1}{n+1} \sigma_C(x)^{\frac{n-1}{n+1}} \,dx,
\end{align*}
depends on the Hilbert geometry $(C,d_C)$ and the choice of the volume density $\sigma_C$. 
Let $\cK_{(0)}(\rn)$ be the set of convex bodies in $\rn$ containing the origin in their interiors.
For $C\in \cK_{(0)}(\rn)$ and $\lambda <1$, the floating area 
$\Omega_C(\lambda C)$ is a centro-affine (or $\gln$) invariant by the definition of floating area and 
the projective invariance of the volume $\vol_C$ (however, note that $\Omega_C(\lambda C)$ is not a projective invariant). 
For the limiting case $\lambda\to1$ and the Busemann floating area, we obtain the following result. 
The proof is based on results by Berck, Bernig, and Vernicos \cite{BBV:2010}, 
who studied the limiting behavior of the volume entropy of $\lambda C$.

\begin{theorem}\label{thm:omegacp}
	For $C\in\cK_{(0)}(\rn)$ with $C^{1,1}$ boundary, 
	\begin{align*}
		\Omega_{n}(C)
			&= 2^{\frac{n-1}2} \lim_{\lambda \to 1-} \Omega_C(\lambda  C)  (1-\lambda)^{\frac{n-1}2},
	\end{align*}
	where $\Omega_C$ is the Busemann floating area.
\end{theorem}

\noindent
Here $\Omega_{n}(C)$ is the classical {\em centro-affine surface area} of $C$ which is defined as
\begin{align*}
	\Omega_{n}(C)
		&=\int_{\bd C} \frac{ H_{n-1}(C,x)^\frac{1}{2}}{ \big(x\cdot n_C(x)\big)^{\frac{n-1}2}} \,dx.
\end{align*}
Centro-affine surface area is an upper semicontinuous and $\gln$ invariant valuation on $\cK^n_{(0)}(\rn)$. Moreover, it is basically the only such functional (see \cite{LR:2010}). For more information on centro-affine surface area, which is also called $L^n$-affine surface area, see \cite{Lutwak:1996, SW:2004, MW:2000}.

\begin{proof}
	Berck, Bernig, and Vernicos \cite[Proposition 2.8]{BBV:2010} obtained that 
	\begin{equation}\label{eqn:berck}
		\lim_{\lambda \to 1-} \sigma_C(\lambda x) (1-\lambda)^{\frac{n+1}2}
			= \frac{H_{n-1}(C, x)^{\frac12}}{\big(2\,x\cdot n_C(x)\big)^\frac{n+1}2},
	\end{equation}
	for $x\in\bd C$. Using a version of Blaschke's rolling theorem, they also showed in \cite[Proposition 2.10]{BBV:2010} that
	\begin{equation}\label{eqn:roll}
		\sigma_C(\lambda x) \le c (1-\lambda)^{-\frac{n+1}2},
	\end{equation}
	where the constant $c$ does not depend on $x$ and $\lambda$.
	Thus, 
	\begin{align*}
		\lim_{\lambda \to 1-} \Omega_C(\lambda C) (1-\lambda)^{\frac{n-1}2}
			&= \lim_{\lambda \to 1-} \lambda^{n\frac{n-1}{n+1}} 
				\int_{\bd C}  H_{n-1}(C,x)^\frac{1}{n+1} 
				\Big((1-\lambda)^{\frac{n+1}2}  \sigma_C(\lambda x)\Big)^\frac{n-1}{n+1} \,dx \\
 			&=  \int_{\bd C}  H_{n-1}(C,x)^\frac{1}{n+1}  
 				\Bigg( \frac{H_{n-1}(C, x)^{\frac12}}{\big(2\,x\cdot n_C(x)\big)^\frac{n+1}2}\Bigg)^\frac{n-1}{n+1} \,dx\\
 			&= 2^{-\frac{n-1}2} \Omega_n(C),
	\end{align*}
	where the last inequality uses Lebesgue's Dominated Convergence Theorem and (\ref{eqn:roll}).
\end{proof}

Theorem \ref{thm:omegacp} holds true not only for the Busemann volume, but also for other notions of volume. This follows, since according to Berck, Bernig, and Vernicos \cite{BBV:2010}, equation \eqnref{eqn:berck} holds true for the volume densities of all volumes that satisfies the following very general assumptions:
\begin{itemize}
	\item The volume measure $\vol_C$ is a Borel measure on $\Int C$ and absolutely continuous with respect to the Lebesgue measure.
	\item If $A\subset C \subset C'$ where $C,C'\in\cK(\R^n)$, then $\vol_C(A)\geq \vol_{C'}(A)$.
	\item If $C$ is an ellipsoid, then $\vol_C$ is the hyperbolic volume.
\end{itemize}
All volume measures that satisfy these conditions are equivalent, i.e., if $\sigma_C$ and $\bar{\sigma}_C$ are the volume densities of two volume measures $\vol_C$ and $\bar{\vol}_C$, then there exist positive real constants $a,b$ such that 
\begin{align*}
	a \sigma_C(x) \leq \bar{\sigma}_C(x) \leq b \sigma_C(x),
\end{align*}
see e.g.\ \cite[p.~249]{HandbookHilbert:2014}.
Hence, by \eqnref{eqn:roll}, we conclude that
\begin{align*}
	\bar{\sigma}_C(\lambda x) \le bc (1-\lambda)^{-\frac{n+1}{2}}.
\end{align*}
Therefore Theorem \ref{thm:omegacp} also holds for any volume measure that satisfies these conditions and
in particular for the Holmes-Thompson volume.

\section{Proof of Theorem \ref{thm:wfb}}\label{sec:proof}

The first step of the proof  is the following disintegration result, which follows easily from the area formula (see e.g.\ \cite[Prop.~3.7]{BW:2016} or \cite[Lem.~4.2]{BFH:2010} for related results).

\begin{lemma}\label{lem:coneformula}
	Let $K,L$ be convex bodies such that $L\subseteq K$ and $0\in\Int L$. For $x\in\bd K$,  
	\begin{align*}
		\Psi(K)-\Psi(L) 
			&= \int_{\bd K} n_K(x)\cdot (x\|x\|^{-n}) \int_{\|x_L\|}^{\|x\|} \psi(tx\|x\|^{-1})t^{n-1}\, dt\, dx,
	\end{align*}
	where  $\{x_L\} = \bd L \cap [0,x]$. 
\end{lemma}

The next step is to give upper and lower bounds of the weighted floating body $K^\phi_\delta$ by a re-parametrized Euclidean floating body. To be more precise, we find $\delta_1=\delta_1(\delta)$ and $\delta_2=\delta_2(\delta)$ such that $0<\delta_1\leq \delta_2$ and $K_{\delta_2} \subseteq K_{\delta}^\phi \subseteq K_{\delta_1}$. Before we go into the details of this proof, we need to fix a few notions.

For $v\in\S^{n-1}$ and $t\in \R$,  define, as before, the closed halfspaces
$H^-(v,t):=\{y\in\R^n : y\cdot v \leq t\}$ and $H^+(v,t) := H^-(-v,-t)$.
The weighted floating body $K_\delta^\phi$ can be expressed as
\begin{align}\label{eqn:floating_par}
	K_\delta^\phi = \bigcap \Big\{H^-\big(v,t_{\delta}(v)\big) : v\in\S^{n-1}\Big\},
\end{align}
where $t_\delta(v)=t(K,\phi,\delta,v)$ is determined implicitly by
\begin{align}\label{eqn:floating_par_t}
	\delta 
		= \Phi\Big(K\cap H^+\big(v,t_{\delta}(v)\big)\Big) 
		= \int_{t_\delta(v)}^{h_K(v)} \int_{K\,\cap\,H(v,s)} \phi(x) \,d\lambda_{H(v,s)}(x)\, ds.
\end{align}
Here $\lambda_{H(v,s)}$ is the Lebesgue measure in the affine hyperplane $H(v,s)=\{y\in\R^n:y\cdot v = s\}$.
Note that there exists $\delta_0>0$ such that the function $t_\delta(v)$ is continuous for $(\delta,v)\in [0,\delta_0)\times \S^{n-1}$ and $t_0(v) = h_K(v)$.

\begin{lemma}
	Let $K\in\cK(\R^n)$ and $\varepsilon\in(0,\min_{\bd K} \phi)$.
	For
	\begin{align}\label{eqn:par_bound}
		\alpha &:= \min_{\bd K} \phi-\varepsilon, 
		& 
		\beta  &:= \max_{\bd K} \phi+\varepsilon,
	\end{align}
	there exists $\delta_0 =\delta_0(\varepsilon)>0$ such that for all $\delta\in(0,\delta_0)$, we have
	\begin{align*}
		K_{\delta/\alpha} \subseteq K_\delta^\phi \subseteq K_{\delta/\beta}.
	\end{align*}
\end{lemma}
\begin{proof}
	Note that by our assumptions $\phi$ is continuous and positive on $\bd K$ and therefore $\min_{\bd K}\phi>0$.
	First we show that there is $\delta_1 = \delta_1(\varepsilon)>0$ such that for all $\delta \in(0,\delta_1)$ and $v\in\S^{n-1}$ we have
	\begin{equation}\label{eqn:proofbound1}
		K\cap H^+\big(v,t_\delta(v)\big) \subseteq \big\{ x\in K : \phi(x)\leq \beta\big\}.
	\end{equation}
	Assume the opposite. Then for all $\delta>0$ there exists $v(\delta)\in\S^{n-1}$ and $y(\delta)\in K$ such that $\phi(y(\delta)) \geq \beta$ and
	$y(\delta)\cdot v(\delta)\geq t_\delta(v(\delta))$.
	By compactness there are converging subsequences with limits $v_0\in\S^{n-1}$ and $y_0\in\bd K$ such that $\phi(y_0)\geq \beta$ and
	$y_0\cdot v_0 \geq t_0(v_0) = h_K(v_0)$.
	Thus $y_0\in\bd K$ and therefore $\phi(y_0) \leq \max_{\bd K} \phi < \beta \leq \phi(y_0)$ -- a contradiction.

	By (\ref{eqn:proofbound1}), we have that
	\begin{align*}
		\delta 
			&= \Phi\Big(K\cap H^+\big(v, t_\delta(v)\big)\Big) 
			\leq \beta \lambda_n\Big(K\cap H^+\big(v,t_\delta(v)\big)\Big),
	\end{align*}
	which yields $t(K,1,\delta/\beta,v)\geq t(K,\phi,\delta,v)$. Thus, by \eqnref{eqn:floating_par} and \eqnref{eqn:floating_par_t}, $K^\phi_\delta\subseteq K_{\delta/\beta}$.
	
	Conversely, there is $\delta_2 = \delta_2(\varepsilon)>0$ such that for all $\delta\in(0,\delta_2)$ and $v\in\S^{n-1}$ we have
	\begin{align*}
		K\cap H^+\big(v,t(K,\phi,\delta,v)\big) \subseteq \big\{x\in K: \phi(x) \geq \alpha\big\}.
	\end{align*}
	Similar to the above we first have
	\begin{align*}
		\delta 
			&= \Phi\Big(K\cap H^+\big(v,t_\delta(v)\big)\Big) 
			\geq \alpha \lambda\Big(K\cap H^+\big(v,t_\delta(v)\big)\Big),
	\end{align*}
	and therefore $K_{\delta/\alpha} \subseteq K^\phi_{\delta}$. Setting $\delta_0 = \min\{\delta_1,\delta_2\}$ concludes the proof.
\end{proof}

For two distinct points $x,y\in\R^n$ the affine segment joining $x$ and $y$ is denoted by $[x,y]$. The previous lemma immediately implies the following result.

\begin{corollary}\label{cor:bound1}
	Let $K\in\cK(\R^n)$, let $\alpha,\beta$ as in \eqref{eqn:par_bound}, and let $z\in\Int\, K$. For $x\in\bd K$ we set 
	\begin{align*}
		\big\{x_{\delta/\alpha}\big\} &= \bd K_{\delta/\alpha} \cap [x,z],&
		\big\{x_{\delta/\beta}\big\}  &= \bd K_{\delta/\beta} \cap [x,z],&
		\big\{x_{\delta}^{\phi}\big\} &= \bd K_{\delta}^\phi \cap [x,z].
	\end{align*}
	Then for $\delta>0$ sufficiently small, we have
	\begin{align*}
		\big\|x_{\delta/\alpha}-z\big\| \leq \big\|x_\delta^\phi-z\big\| \leq \big\|x_{\delta/\beta}-z\big\|.
	\end{align*}
\end{corollary}

To complete the proof, we proceed as follows: The left hand-side of \eqnref{eqn:limit} can be written as an integral over $\bd K$ by Lemma \ref{lem:coneformula}. Theorem \ref{thm:wfb} follows by applying Lebesgue's Dominated Convergence Theorem and calculating the point-wise limit of the integrand. To do so, we need to bound the integrand from above by an integrable function.

We denote by $r_K\colon\bd K \to [0,+\infty)$ the maximal radius of a Euclidean ball that contains $x\in\bd K$ and is contained in $K$. It was proven in \cite{SW:1990}, that for $\alpha>-1$ we have
\begin{align*}
	\int_{\bd K} r_K(x)^{\alpha}\,dx <+\infty.
\end{align*}
Hence $r_K$ is an integrable function and  it was already used as upper bound of the integrand for the Euclidean floating body. The following upper bound for the weighted floating body follows by the Euclidean results obtained in \cite{SW:1990}.

\begin{lemma}\label{lem:asym2}
	Let $K\in\cK(\R^n)$ with  $0\in\Int\, K$. There exists $C>0$ such that for $\delta>0$ sufficiently small 
	\begin{align*}
		\frac{x\cdot n_K(x)}{\delta^{2/(n+1)}\|x\|^n} \int_{\|x_\delta^\phi\|}^{\|x\|} t^{n-1} \psi(tx/\|x\|)\, dt 
			&\leq C\, \big(\max_{K} \psi\big)\, r_K(x)^{-\frac{n-1}{n+1}},
	\end{align*}
	for almost all $x\in \bd K$.
\end{lemma}
\begin{proof}
	Since $0\in\Int K$, by Corollary \ref{cor:bound1} we have $\|x_{\delta}^\phi\| \geq \|x_{\delta/\alpha}\|$. We conclude
	\begin{align*}
		\frac{x\cdot n_K(x)}{\delta^{2/(n+1)}\|x\|^n} \int_{\|x_\delta^\phi\|}^{\|x\|} t^{n-1} \psi(tx/\|x\|)\, dt
			&\leq \big(\max_K \psi\big) \, \frac{x\cdot n_K(x)}{\delta^{2/(n+1)}\|x\|^n}
				\int_{\|x_{\delta/\alpha}\|}^{\|x\|} t^{n-1} \, dt.
	\end{align*}
	Furthermore,
	\begin{align*}
		\frac{x\cdot n_K(x)}{\delta^{2/(n+1)}\|x\|^n} \int_{\|x_{\delta/\alpha}\|}^{\|x\|} t^{n-1} \, dt
			&\leq \, \frac{x\cdot n_K(x)}{\|x\|} \frac{\big\|x-x_{\delta/\alpha}\big\|}{\delta^{2/(n+1)}}
			\leq C r_K(x)^{-\frac{n-1}{n+1}},
	\end{align*}
	where the last inequality is the Euclidean result established in \cite[Lemma 6]{SW:1990}.
\end{proof}

To calculate the point-wise limit of the integrand, we also use the Euclidean result to obtain the result for the weighted floating body. We recall some notions for boundary points of a convex body (see, for example, \cite[Section 2.2, Section 2.5]{Schneider:2014}).

A boundary point $x$ of $K$ is called {\em regular} if there is a unique outer unit normal $n_K(x)$ to $K$ at $x$. Almost all boundary points are regular.
Recall that for a convex body $K$ the boundary $\bd K$ is  $C^2$ almost everywhere in the following sense:
If $x$ is a regular boundary point, there is $\varepsilon>0$ and an open neighborhood $U$ of $x$ such that $U\cap \bd K$ can be described as
\begin{align*}
	U\cap \bd K = \big\{x+v-f(v)\,n_K(x):v\in n_K(x)^\bot \cap \varepsilon \,\B^n\big\},
\end{align*}
where $f\colon n_K(x)^\perp\cap \varepsilon\,\B^n\to \R$ is a convex function which satisfies $f\geq 0$, $f(0)=0$ and $n_K(x)^\perp=\{y\in\R^n: y\cdot n_K(x)=0\}$. A regular boundary point $x\in\bd K$ is \emph{normal} (or second order differentiable), if $f$ is twice differentiable at $0$ in the following sense: $f$ is differentiable at $0$ and there exists a symmetric linear map $A\colon \R^n\to \R^n$ such that for $v,w\in n_K(x)^\perp$,
\begin{align*}
	f(w) = f(v)+ \nabla f(v)\cdot (w-v) + \frac{1}{2} A(w-v)\cdot (w-v) + o\big(\|w-v\|^2\big),
\end{align*}
as $\| w-v\|\to 0$.
Note that almost all boundary points are normal (see \cite[Thm.~2.5.5]{Schneider:2014}), and the (generalized) Gauss--Kronecker curvature $H_{n-1}(K,x)=\det(A)$ exists for normal boundary points.

\begin{lemma}\label{lem:lim2}
	Let $K \in \cK(\R^n)$. If $x\in\bd K$ is a normal boundary point, then
	\begin{align*}
		\lim_{\delta\to 0^+} \frac{x\cdot n_K(x)}{\delta^{2/(n+1)}\|x\|^n} \int_{\|x_\delta^\phi\|}^{\|x\|} t^{n-1} \psi(tx/\|x\|)\, dt
			&= \alpha_n\, H_{n-1}(K,x)^{\frac{1}{n+1}}\, \phi(x)^{-\frac{2}{n+1}}\, \psi(x).
	\end{align*}
\end{lemma}

In the proof of Lemma \ref{lem:lim2} we will use the following two results.

\begin{lemma}[\!\! {\cite[Lemma 2.9]{BW:2016}}]
	Let $K\in\cK(\R^n)$ with $0\in\Int\, K$ and $\varepsilon>0$.
	If $x\in\bd K$ is a normal boundary point such that $H_{n-1}(K,x)>0$, then
	 there is $\delta_0=\delta_0(\varepsilon)$ such that for all $\delta\in(0,\delta_0)$, 
	\begin{align*}
		[x,0]\cap L_\delta^\phi = [x,0] \cap K_\delta^\phi,
	\end{align*}
where  $L=K\cap (x+ \varepsilon\, \B^n)$.
	In particular, if we set $\{x_\delta^{\phi,K}\} = \bd K_\delta^\phi\cap [x,0]$ and $\{x_\delta^{\phi,L}\} = \bd L_\delta^\phi\cap [x,0]$,
	then $x^{\phi,K}_\delta=x^{\phi,L}_\delta$.
\end{lemma}

\begin{lemma}[\!\!{\cite{SW:1990}}]\label{lem:lim1}
	Let $K \in\cK(\R^n)$. If $x\in\bd K$ is a normal boundary point, then
	\begin{align*}
		\lim_{\delta\to 0^+} \frac{x\cdot n_K(x)}{\delta^{2/(n+1)}\|x\|^n} \int_{\|x_\delta\|}^{\|x\|} t^{n-1} \, dt
			&= \alpha_n\, H_{n-1}(K,x)^{\frac{1}{n+1}}.
	\end{align*}
\end{lemma}

\begin{proof}[Proof of Lemma \ref{lem:lim2}]
	Since $x$ is normal, $H_{n-1}(K,x)$ exists. First, if $H_{n-1}(K,x)=0$, then
	\begin{align*}
		\frac{x\cdot n_K(x)}{\delta^{2/(n+1)}\|x\|^n} \int_{\|x_\delta^\phi\|}^{\|x\|} t^{n-1} \psi(tx/\|x\|)\, dt
			&\leq \big(\max_K \psi\big)\,
				\frac{x\cdot n_K(x)}{\delta^{2/(n+1)}\|x\|^n} \int_{\|x_{\delta/\alpha}\|}^{\|x\|} t^{n-1}\, dt.
	\end{align*}
	By Lemma \ref{lem:lim1}, we conclude
	\begin{align*}
		\limsup_{\delta\to 0^+} \frac{x\cdot n_K(x)}{\delta^{\frac{2}{n+1}}\|x\|^n} \int_{\|x_\delta^\phi\|}^{\|x\|} t^{n-1} \psi\Big(\frac{tx}{\|x\|}\Big)\, dt
			& \leq \frac{\max_K \psi}{\alpha^{\frac{2}{n+1}}} 
				\limsup_{\delta\to 0^+} \frac{x\cdot n_K(x)}{(\delta/\alpha)^{\frac{2}{n+1}}\|x\|^n} 
				\int_{\|x_{\delta/\alpha}\|}^{\|x\|} t^{n-1}\, dt
			=0.
	\end{align*}
	
	Now assume $H_{n-1}(K,x)>0$ and let $\varepsilon>0$ be arbitrary. Set $L = K\cap (x+\varepsilon\, \B^n)$. Then $H_{n-1}(L,x) = H_{n-1}(K,x)$ and $n_K(x)=n_L(x)$. Furthermore, for $\delta$ small enough, we have
	\begin{align*}
		\frac{x\cdot n_K(x)}{\delta^{2/(n+1)}\|x\|^n} \int_{\|x_\delta^{\phi,K}\|}^{\|x\|} t^{n-1} \psi(tx/\|x\|)\, dt
			&= \frac{x\cdot n_L(x)}{\delta^{2/(n+1)}\|x\|^n} \int_{\|x_\delta^{\phi,L}\|}^{\|x\|} t^{n-1} \psi(tx/\|x\|)\, dt.
	\end{align*}
	Let $z\in (\Int L)\cap [0,x]$. We apply Corollary \ref{cor:bound1} on $L$ with $\varepsilon$ and obtain, for $y\in\partial L$, 
	\begin{align*}
		\|y_{\delta/\beta}-z\| \geq \|y_{\delta}^{\phi,L}-z\| \geq \|y_{\delta/\alpha}-z\|,
	\end{align*}
	where $\beta = \max_{\partial L} \phi +\varepsilon $ and $\alpha = \min_{\partial L} \phi - \varepsilon$.
	Since $\|x\| = \|z\|+\|x-z\|$, this yields $\|x_{\delta/\beta}\| \geq \|x_{\delta}^{\phi,L}\| \geq \|x_{\delta/\alpha}\|$. We conclude
	\begin{align*}
		\frac{x\cdot n_L(x)}{\delta^{\frac{2}{n+1}}\|x\|^n} \int_{\|x_\delta^{\phi,L}\|}^{\|x\|} t^{n-1} \psi\Big(\frac{tx}{\|x\|}\Big)\, dt
			&\leq \frac{x\cdot n_L(x)}{\delta^{\frac{2}{n+1}}\|x\|^n} \,
				\Bigg(\max_{t\in \big[\|x_{\delta/\alpha}^L\|,\|x\|\big]} \psi\Big(\frac{tx}{\|x\|}\Big)\Bigg) 
				\int_{\|x_{\delta/\alpha}^L\|}^{\|x\|} t^{n-1} \, dt,
	\end{align*}
	and therefore
	\begin{align*}
		\limsup_{\delta\to 0^+} \frac{x\cdot n_L(x)}{\delta^{2/(n+1)}\|x\|^n} \int_{\|x_\delta^{\phi,L}\|}^{\|x\|} t^{n-1} \psi(tx/\|x\|)\, dt
			&\leq \alpha_n \, H_{n-1}(L,x)^{\frac{1}{n+1}}\, \frac{\psi(x)}{\alpha^{2/(n+1)}}.
	\end{align*}
	Conversely, we have
	\begin{align*}
		\frac{x\cdot n_L(x)}{\delta^{\frac{2}{n+1}}\|x\|^n} \int_{\|x_\delta^{\phi,L}\|}^{\|x\|} t^{n-1} \psi\Big(\frac{tx}{\|x\|}\Big)\, dt
			&\geq \frac{x\cdot n_L(x)}{\delta^{\frac{2}{n+1}}\|x\|^n}\,
				\Bigg(\min_{t\in \big[\|x_{\delta/\beta}^L\|,\|x\|\big]} \psi\Big(\frac{tx}{\|x\|}\Big)\Bigg)
				\int_{\|x_{\delta/\beta}^L\|}^{\|x\|} t^{n-1} \, dt,
	\end{align*}
	and hence
	\begin{align*}
		\liminf_{\delta\to 0^+} \frac{x\cdot n_L(x)}{\delta^{2/(n+1)}\|x\|^n} \int_{\|x_\delta^{\phi,L}\|}^{\|x\|} t^{n-1} \psi(tx/\|x\|)\, dt
			&\geq \alpha_n\, H_{n-1}(L,x)^{\frac{1}{n+1}} \frac{\psi(x)}{\beta^{2/(n+1)}}.
	\end{align*}
	Since $\varepsilon>0$ can be chosen arbitrarily small and $\beta,\alpha \to \phi(x)$ for $\varepsilon\to 0$, we conclude
	\begin{align*}
		\lim_{\delta\to 0^+} \frac{x\cdot n_L(x)}{\delta^{2/(n+1)}\|x\|^n} \int_{\|x_\delta^{\phi,L}\|}^{\|x\|} t^{n-1} \psi(tx/\|x\|)\, dt
			&= \alpha_n\, H_{n-1}(L,x)^{\frac{1}{n+1}}\, \phi(x)^{-\frac{2}{n+1}}\, \psi(x).
	\end{align*}
	This finishes the proof, as, for $\delta>0$ sufficiently small, we have $n_L(x) = n_K(x)$, $H_{n-1}(L,x) = H_{n-1}(K,x)$ and $x_\delta^{\phi,L} = x_\delta^{\phi,K}$.
\end{proof}

The proof of \eqnref{eqn:limit} is now straightforward. By Lemma \ref{lem:coneformula} we have
\begin{align*}
	\frac{\Psi(K)-\Psi(K^\phi_\delta)}{\delta^{\frac{2}{n+1}}} = \int_{\bd K}  \frac{x\cdot n_K(x)}{\delta^{(n+1)/2}\|x\|^n} \int_{\|x_\delta\|}^{\|x\|} t^{n-1} \psi(tx/\|x\|)\, dt\, dx.
\end{align*}
By Corollary \ref{cor:bound1}, there is $\delta_0>0$ such that the integrand is bounded by an integrable function for all $\delta<\delta_0$. By Lebesgue's Dominated Convergence Theorem and Lemma \ref{lem:lim2}, we conclude
\begin{align*}
	\lim_{\delta\to 0^+} \frac{\Psi(K)-\Psi(K^\phi_\delta)}{\delta^{\frac{2}{n+1}}} = \alpha_n \int_{\bd K} H_{n-1}(K,x)^{\frac{1}{n+1}} \phi(x)^{-\frac{2}{n+1}} \psi(x)\, dx.
\end{align*}

\subsection*{Acknowledgments}
The authors thank Juan Carlos Alvarez Paiva and Matthias Reitzner for helpful discussions. 
The work of Monika Ludwig was supported, in part, by Austrian Science Fund (FWF) Project P25515-N25.  Elisabeth Werner was partially supported by NSF grant 1504701.

\goodbreak

\medskip
\parindent=0pt

\bigskip
\begin{samepage}
	Florian Besau\\
	Institut f\"ur Mathematik\\
	Goethe-Universit\"at Frankfurt\\
	Robert-Mayer-Str.~10\\
	60054 Frankfurt, Germany\\
	e-mail: besau@math.uni-frankfurt.de
\end{samepage}

\bigskip
\begin{samepage}
Monika Ludwig\\
Institut f\"ur Diskrete Mathematik und Geometrie\\
Technische Universit\"at Wien\\
Wiedner Hauptstra\ss e 8-10/1046\\
1040 Wien, Austria\\
e-mail: monika.ludwig@tuwien.ac.at
\end{samepage}

\bigskip
\begin{samepage}
Elisabeth M.~Werner\\
Department of Mathematics, Applied Mathematics and Statistics\\
Case Western Reserve University\\
10900 Euclid Avenue\\
Cleveland, Ohio 44106, USA\\
e-mail: elisabeth.werner@case.edu
\end{samepage}
\bigskip

\end{document}